\documentclass[11pt,a4paper]{article}
\usepackage{comment}
\usepackage{amssymb}
\usepackage{amsmath}
\usepackage{amsthm}
\usepackage{amsfonts}
\usepackage{url}
\usepackage{hyperref}
\usepackage{breakurl}
\usepackage{graphicx}

\newtheorem{theorem}{Theorem}

\newtheorem{corollary}[theorem]{Corollary}
\newtheorem{lemma}[theorem]{Lemma}

\newtheorem{definition}{Definition}

\newtheorem{remark}{Remark}

\newcommand{\bi}{\begin{itemize}}
\newcommand{\ei}{\end{itemize}}
\newcommand{\be}{\begin{equation}}
\newcommand{\ee}{\end{equation}}
\newcommand{\ben}{\begin{equation*}}
\newcommand{\een}{\end{equation*}}
\newcommand{\bea}{\begin{eqnarray}}
\newcommand{\eea}{\end{eqnarray}}
\newcommand{\bean}{\begin{eqnarray*}}
\newcommand{\eean}{\end{eqnarray*}}
	
\newcommand{\pos}{\mathbb{N}_1} 
\newcommand{\integer}{\mathbb{Z}}

\newcommand{\diag}{\mathrm{diag}}
\newcommand{\dotdot}{..}
\begin{document}
\bibliographystyle{plain}

\title{Maximal Determinants and Saturated D-optimal Designs 
				of Orders 19 and 37}

\author{Richard P. Brent\\
Mathematical Sciences Institute,\\
Australian National University, \\
Canberra, ACT 0200, Australia\\
{\tt maxdet@rpbrent.com}
\and
William Orrick\\
Department of Mathematics,\\
Indiana University,\\
Bloomington, IN 47405, USA\\
{\tt worrick@indiana.edu}
\and 
Judy-anne Osborn\\
University of Newcastle,\\
Callaghan, NSW 2308, Australia\\
{\tt Judy-anne.Osborn@anu.edu.au}
\and 
Paul Zimmermann\\
INRIA Nancy -- Grand Est,\\ 
Villers-l\`es-Nancy, 
France\\
{\tt Paul.Zimmermann@inria.fr}
}

\date{\today}

\maketitle

\begin{abstract}
A saturated D-optimal design is a $\{+1,-1\}$ square matrix of given order
with maximal determinant.  We search for saturated D-optimal designs of
orders $19$ and $37$, and find that known matrices due to
Smith, Cohn, Orrick and Solomon are optimal.
For order $19$ we find all inequivalent saturated D-optimal
designs with maximal determinant, 
$2^{30}\times 7^2 \times 17$, 
and confirm
that the three known designs comprise a complete set.  For order $37$
we prove that the maximal determinant is 
$2^{39}\times 3^{36}$, 
and find a sample of inequivalent saturated D-optimal designs.
Our method is an
extension of that used by Orrick to resolve the previously smallest
unknown order of $15$; and by Chadjipantelis, Kounias and Moyssiadis
to resolve orders $17$ and $21$.
The method is a two-step computation which first searches
for candidate Gram matrices and then attempts to decompose them.
Using a similar method, 
we also find the complete spectrum of determinant values for
$\{+1,-1\}$ matrices of order~$13$.
\end{abstract}

\section{Introduction} 
The Maximal Determinant problem of Hadamard \cite{Hadamard93,maxdet}
asks for the largest possible determinant of an $n \times n$ matrix whose
entries are drawn from the set $\{+1,-1\}$.
We are only interested in the absolute value of the determinant, since we
can always change the sign of the determinant by changing the sign of a row.  
The problem in its full generality has been open since first posed by
Hadamard~\cite{Hadamard93}, and has applications to areas such as
Experimental Design and Coding Theory.

We could equally well consider $\{0,1\}$ matrices. There is a well-known
mapping~\cite{Osborn02} from $\{0,1\}^{(n-1)\times (n-1)}$ matrices to
$\{+1,-1\}^{n \times n}$ matrices 
which multiplies the determinant by $(-2)^{n-1}$, and 
vice versa. To avoid confusion we only consider $\{+1,-1\}$ matrices.
Their determinants are always divisible by $2^{n-1}$,
thanks to the correspondence with $\{0,1\}$ matrices.
Thus, it is convenient to let $D_n$ denote $\max |\det(R)|$,
where the maximum is over all $\{+1,-1\}^{n \times n}$ matrices $R$,
and $d_n = D_n/2^{n-1}$.

There is an extensive literature on the Maximal Determinant problem, which
splits into four cases, according to the value of $n \bmod 4$.  A general
upper bound of
\be \label{eq:Hadamard} 
D_n \le n^{n/2}
\ee
on the maximal determinant applies to all the four
cases, but is not achievable unless
$n = 1, 2$, or $n \equiv 0 \bmod 4$.
The conjecture that this bound is always achievable
when $n \equiv 0 \pmod 4$ is
known as the \emph{Hadamard Conjecture}, and has been the subject of much
investigation, see for example~\cite{GeSe79,Horadam07}.  
Smaller upper bounds are known for each of the other three equivalence
classes modulo four.

A bound which holds for all odd orders, and which is known to be
sharp for an infinite sequence of orders congruent to $1 \pmod 4$, 
is 
\be \label{eq:eb_bound}
D_n \le \sqrt{(n-1)^{n-1}(2n-1)},
\ee
due independently to Ehlich~\cite{Ehlich64a} 
and Barba \cite{Bar33}.
A smaller upper bound, due to Ehlich~\cite{Ehlich64b}, 
applies only in the case $n \equiv 3 \pmod 4$:
\be \label{eq:e_bound} 
D_n \le \sqrt{(n-3)^{n-s}(n-3+4r)^{u}(n+1+4r)^{v}
\left(1 - \tfrac{ur}{n-3+4r} - \tfrac{v(r+1)}{n+1+4r}\right)}\;.
\ee
Here $s=3$ for $n=3$ (and the factor $(n-3)^{(n-s)}$ is interpreted
as $1$ in this case), 
$s=5$ for $n=7$, 
$s=6$ for $11 \le n \le 59$, 
$s=7$ for $n \ge 63$,
$r = \lfloor n/s \rfloor$, $v = n - rs$,  and $u = s - v$.
The complicated form of the bound~(\ref{eq:e_bound})
as compared with~(\ref{eq:eb_bound}) is indicative of the extra
difficulties which often seem to arise when $n \equiv 3 \pmod 4$.  The
bound~(\ref{eq:e_bound}) is sharp when $n=3$; it is not known if 
it is sharp for any $n > 3$.

In this work we settle the smallest hitherto unresolved case of $n~=~19$.
This case has remained open despite higher orders (for example, $21$) being
solved by similar methods, mainly because the use of~(\ref{eq:eb_bound}) and
its generalisation when the Gram matrix has a fixed block---see
Theorem~\ref{thm:enhancedKM}---is much more effective in pruning the search
tree than are~(\ref{eq:e_bound}) and its generalisations. All orders
congruent to 3 mod~4 and larger than 19 are currently open.

In this paper we only consider odd orders.
The smallest unresolved orders which are congruent to 1 mod 4 
are $n=29, 33$ and $37$.  Of these, we resolve $n=37$, and improve the
upper bounds for $n = 29, 33$.  
For a summary, see Table~1 in~\S\ref{sec:other_bounds}.

Our method is structurally similar to that used for $n = 15$ by
Orrick~\cite{Orrick05}, and by earlier authors for $n = 17$ in~\cite{MoKo82}
and $n = 21$ in~\cite{ChKoMo87}. There are two essential steps, Gram
finding and decomposition. In the cases we consider, decomposition by hand
would be tedious for $n = 19$, and infeasible for larger orders. Thus, we
implement a back-tracking computer search to deal with this second step,
describing such an algorithm for the first time in the literature.

Our Gram-finding algorithm is discussed in~\S\ref{sec:gram-finding},
and our decomposition algorithms in~\S\ref{sec:decomp}.
The results for order~$19$ are described in~\S\ref{sec:order19},
and for order~$37$ in~\S\ref{sec:order37}.
In~\S\ref{sec:other_bounds} we give some new upper and lower
bounds for various odd orders.  For
orders $n = 29$, $33$, $45$, $49$, $53$ and $57$ we have not been able to
determine $D_n$ precisely, but we have reduced the gap between the known
upper and lower bounds.
Finally, in~\S\ref{sec:spectrum} we also find the complete spectrum of 
determinant values for $\{+1,-1\}$ matrices of order~$13$.
Previously, the spectrum was only known for orders up to~$11$.

\section{Definitions}
  
$\integer$ denotes the integers,
and $\pos$ the positive integers.
The following definitions are largely taken from \cite{Orrick05}, to which
we refer for further technical definitions.

\begin{definition}\label{def:design} 
A \emph{design} is an $m \times n$ matrix with entries drawn from the set
$\{+1, -1\}$.  If $m=n$ the design is called \emph{saturated}.  If the
absolute value of the 
determinant of the saturated design is maximal for its order, the design is
called \emph{D-optimal}.
\end{definition}

In this paper we consider saturated D-optimal designs of odd order.  It is
convenient to consider ``normalized'' designs, leading to the next
definition:

\begin{definition}
A vector with elements
in $\{+1,-1\}$ is {\em parity normalized} iff it has an even number of
positive elements. A design is {\em parity normalized}
iff all its rows and columns are parity normalized.  
\end{definition}

It is easy to show, as
in~\cite[Lemmas 3.1, 3.2]{Ehlich64a}, that any saturated design of odd
order can be converted to a unique parity normalized matrix by a series of
negations of rows and columns.

If $R_1$ is a design, then any signed permutation of the rows and columns
of $R_1$ gives another design $R_2$, which we can regard as equivalent to the
original design since $|\det(R_1)| = |\det(R_2)|$.
We can also change the signs of any rows and/or
columns of without changing more than the sign of the determinant.
This suggests the following definition, in which a signed permutation matrix 
is a permutation of the rows or columns of a diagonal matrix 
$\diag(\pm 1, \pm 1, \ldots, \pm 1)$.

\begin{definition} \label{def:Hequiv}
Two designs $R$ and $S$ are {\em Hadamard equivalent}
iff $S = PRQ$ for some pair of signed permutation matrices $(P,Q)$.
\end{definition}

\begin{definition} \label{def:gram} 
If $R$ is a design, then
$G=RR^T$ is called the \emph{Gram matrix} of $R$,
and $H = R^TR$ is called the \emph{dual Gram matrix} of $R$.
\end{definition}

\begin{definition}
Two symmetric matrices $G_1$ and $G_2$ are 
{\em Gram equivalent}
iff $G_1 = PG_2 P^T$ for some signed
permutation matrix $P$.
\end{definition}

\begin{samepage}
\begin{definition} \label{def:candGram} 
Let $d_{\text{min}} > 0$ and let $\mathcal{M}_{n, p}$ be the set
of square matrices $M$, of order $p$, 
$1 \le p \le n$,	
satisfying properties 1--3 below.

\begin{enumerate}
\item  $M$ is symmetric and positive definite;
\item $M_{i,i}=n$;
\item \label{it:n_mod4} $M_{i,j} \equiv n \pmod 4$.
\end{enumerate}

A matrix $M \in \mathcal{M}_{n,p}$ is called a \emph{candidate principal
minor}.  If, furthermore, $n=p$ and the following additional properties 4--5
hold:

\begin{enumerate}
\item[4.]  $\det(M) = d^2$ for $d \in \integer$;
\item[5.]  $d \ge d_{\text{min}}$;
\end{enumerate}
then $M$ is called a \emph{candidate Gram matrix}.
\end{definition}
\end{samepage}

It is clear that Properties 1, 2, 4 and 5 of candidate Gram matrices are
satisfied by all Gram matrices.  Furthermore, Property 3 of candidate Gram
matrices holds for Gram matrices $G = RR^T$ 
if $R$ is assumed to be parity normalized.

\newpage
\section{Gram-finding Algorithm} \label{sec:gram-finding}
 
We summarize our Gram-finding algorithm below.  The method is essentially
that described in greater detail in~\cite{Orrick05}.
We search for candidate Gram matrices whose
determinant is greater than or equal to a positive parameter $d_{\min}^2$.

\begin{enumerate}

\item Set $r = 1$ and start from the candidate principal minor $M_1 = (n)$.
\item \label{it:inc}  Increment $r$. 
Build a list of \emph{admissible} vectors $f$, 
and \emph{allowable} vectors~$\gamma$ (for details see~\cite{Orrick05}). 
\item \label{it:f}For each possible 
lexicographically maximal matrix $M_{r-1}$ of order 
$r - 1$, and each admissible vector $f$, construct the matrix
\begin{equation}
M_r = 
\left [
\begin{matrix}
M_{r-1} & f\\
f^T & n
\end{matrix}
\right ].
\end{equation}
If $r = n$, 
\begin{enumerate}
\item[(a)] if $\det(M_r) = d^2 \ge d_{\min}^2$,
output the candidate Gram matrix $M_r$.
\end{enumerate}
If $r < n$, 
\begin{enumerate}
\item[(b)]	\label{it:gamma}
evaluate
\begin{equation}
d =
\left |			
\begin{matrix}
M_r & \gamma\\
\gamma^T & 1
\end{matrix}
\right |
\end{equation}
for each allowable vector $\gamma$, looking for a ``good $d$'', namely $d$
such that the function $u_r$ 
in Theorem~\ref{thm:enhancedKM} satisfies $u_r(1,d) \ge d_{\min}^2$.
If a good $d$ is found, try to extend $M_r$ by recursively calling
the algorithm (starting at step~\ref{it:inc}).
\end{enumerate}
\end{enumerate}

Pruning at step~\ref{it:gamma}(b) of the above algorithm relies on
the following Theorem, originally used by 
Moyssiadis and Kounias~\cite{MoKo82} 
to find a maximal Gram matrix of order $n=17$.  Our version
below contains a sharper bound~(\ref{eqn:evensharperbound})
applicable when $n \equiv 3 \pmod 4$.

\newpage

\begin{theorem} \label{thm:enhancedKM} \emph{[Enhanced Kounias \& Moyssiadis]}
Let $M = \left [\begin{matrix}M_r &B\\ B^T & A \end{matrix} \right ]$ 
be a symmetric, positive definite matrix of
order $n$ with elements taken from a set $\Phi$ whose members are greater
than or equal in magnitude to some number $c$, 
$0 < c \le n$.  
Here $M_r$ is a
candidate principal minor of order $r \le n$, and $A$ is a square matrix of
order $n-r$, with diagonal elements $A_{i,i}=n$.  
The columns of the $r \times (n-r)$
matrix $B$ are taken from some set $\Gamma_r \subseteq \Phi^r$.  Define
$d^*$ and $\gamma^*$ by
\begin{equation} \label{eq:dstar}
d^* = \left | \begin{matrix}M_r & \gamma^*\\ 
      {\gamma^*}^T & c\end{matrix}\right |
=
\max_{\gamma \in \Gamma_r}
\left | \begin{matrix}M_r & \gamma \\ {\gamma}^T & c\end{matrix}\right |\,.
\end{equation}
Then
\begin{equation} \label{eqn:bound}
\det(M) \le u_r(c,d^*),
\end{equation}
where
\begin{equation*}
u_r(c,d) =  (n - c)^{n-r-1}\left [(n - c)\det(M_r) + 
		(n - r )\max(0, d)\right ].
\end{equation*}
Furthermore, if $n \equiv 3 \pmod 4$, 
then the following bounds apply:
\begin{eqnarray} \label{eqn:sharperbound} 
\det(M) & \le & (n-1)^{n-r}\det(M_r) +   \\ \nonumber
	    && {[}(n-1)^{n-r} - (n-3)^{n-r} - (n-r)(n-3)^{n-r-1}]\max(0,d^*)
\end{eqnarray}
and, assuming $\det M_r > (n-3)\det M_{r-1}$,  
\begin{align} \label{eqn:evensharperbound} 
\renewcommand{\arraystretch}{1.2}
&\det(M) \le \max_k\mathop{\max_{b_1,\ldots,b_k\in\mathbb{N}_1}}_{b_1+
  \ldots+b_k=n-r}
\max_{\gamma_1^*,\ldots,\gamma_k^* \in \Gamma_r}\notag\\
&\qquad\det
\left[\begin{array}{c|ccc}M_r & \gamma_1^*j_{b_1}^T & 
  \cdots & \gamma_k^*j_{b_k}^T\\
\hline
j_{b_1}{\gamma_1^*}^T & (n-3)I_{b_1}+3J_{b_1} &\cdots & -J_{b_1,b_k}\\
\vdots & \vdots  & \ddots & \vdots\\
j_{b_k}{\gamma_k^*}^T & -J_{b_k,b_1} & \cdots & (n-3)I_{b_k}+3J_{b_k}
\end{array}\right]
\renewcommand{\arraystretch}{1}
\end{align}
where $M_{r-1}$ is the principal $(r-1)$-by-$(r-1)$ minor of $M_r$, $j_a$ is
the column vector of dimension $a$ whose elements all equal 1, $J_{a,b}$ is
the $a$-by-$b$ matrix whose elements all equal 1, and $J_a=J_{a,a}$.
\end{theorem}

\begin{proof}
For a proof of inequalities~(\ref{eqn:bound}) and~(\ref{eqn:sharperbound})
we refer to \cite[Theorem 3.1 and Corollary 3.3]{Orrick05}.  A proof
of~(\ref{eqn:evensharperbound}) is sketched in the Appendix.
\end{proof}

The bound~(\ref{eqn:evensharperbound}) is sharp and therefore potentially
much more powerful than~(\ref{eqn:bound}) or~(\ref{eqn:sharperbound}). 
Unfortunately, the multidimensional search for the optimal set of block
sizes $(b_1,\ldots,b_k)$, and the optimal set of vectors, $\{\gamma_j^*\}$
is expensive.  We therefore restricted its use to the situation where the
non-diagonal elements of the last column of $M_r$ all equal $-1$.  This
allows us to assume that all $\gamma_j^*$ consist entirely of elements $-1$,
and we are left only with the search for the optimal partition.  Much of the
computation associated with the latter search need only be done once.  The
use of~(\ref{eqn:evensharperbound}) resulted in an approximately 15\%
improvement in running time.

\newpage

\section{Decomposition Algorithm}	\label{sec:decomp}

The output of the program described in the previous section is a list $\cal L$ of candidate Gram matrices, complete in the sense that it contains one representative of each
Gram equivalence class with determinant $\ge d_{\min}^2$ 
for a given bound $d_{\min}$.  We
need to determine if any $G \in \cal L$ decomposes as a product
$G = RR^T$ for some square $\{+1,-1\}$ matrix $R$.   This section describes several algorithms for carrying out this task.

For each candidate $G$ this
involves a (possibly large) combinatorial search. It may be regarded as
searching a tree, where each level of the tree corresponds to one row
of the matrix $R$.  At level $k$ we know $k-1$ rows of $R$ and try to find a
$k$-th row satisfying the constraints.  Each node at the $k$-th level
corresponds to one possible choice of the $k$-th row of $R$, given the
preceding rows.   If $G=RR^T$ has solutions, then the solution matrices,
$R$, correspond to nodes at level $n$ of the tree.  In principle our
procedure may generate many Hadamard-equivalent solutions.  We
prune the tree so as to limit the number of duplicate solutions produced.

The search algorithm relies on a family of constraints. The zeroth member
of the family is a special case which can be implemented with a single Gram
matrix and we call this constraint the \emph{single-Gram constraint}. The
rest of the family require both the Gram matrix $G=RR^T$ and the dual Gram
matrix $H=R^TR$ and we call these \emph{Gram-pair constraints}.

Our decomposition algorithm differs from that described in~\cite{Orrick05}
in several respects.  First, it builds up $R$ by rows, instead of by rows
and columns simultaneously.  Second, it uses more general Gram-pair
constraints (see~(\ref{eq:block_nonlinear}) with $j \ge 2$ below).

For clarity, we first describe a version of the algorithm which
uses only the single-Gram constraint.

\subsection{Decomposition using only the single-Gram constraint}
\label{subsec:linear}

Denote the elements of $G$ by $g_{i,j}$ for $i,j =1,..,n$,  
and the rows of $R$ by $r_i$ for $i=1,\ldots,n$. 
Then the constraint $RR^T = G$ is equivalent to
\begin{equation}
r_i r_j^T = g_{i,j}				\label{eq:gij_constraint}
\end{equation}
for $1 \le i \le j \le n$.  The search tree is created by application of
these constraints.  
An outline of the basic algorithm is as follows. The main work is
done by a recursive procedure {\tt search}$(k)$ which searches 
an (implicit) subtree at level $k$, where the root is at level~$1$.

\newpage

\subsubsection*{Algorithm using the single-Gram constraint, version 1}

\begin{enumerate}
\item Initialize \emph{level} $k=1$, \emph{first row} 
$r_1= (1,1,\ldots, 1)$ and $R=r_1$.  

\item Call {\tt search}$(k)$.

\item Output ``no solution'' and halt.

\item {\tt search}$(k)$: \label{it:lin}
If $k=n$, output the solution $R$ and halt.  
Otherwise increment $k$. Find all solutions 
$r_k \in \{+1,-1\}^n$
of the (under-determined) 
set of simultaneous linear equations:

\begin{equation*}
\left[
\begin{matrix}
r_1\\r_2\\ \vdots \\r_{k-1}\\
\end{matrix}
\right] r_k^T = \left[
\begin{matrix}
g_{1,k}\\g_{2,k}\\ \vdots \\ g_{k-1,k}\\
\end{matrix} \right]
\end{equation*}
For each solution $r_k$, append $r_{k}$ to $R$ 
and call {\tt search}$(k)$ recursively to search the relevant subtree.
Return to the caller (i.e. backtrack). 
\end{enumerate}

We justify the choice of the first row in the above algorithm by observing
that $G = RR^T$ is invariant under a signed permutation of columns
of $R$, i.e. $R \mapsto RP$ for any signed permutation matrix $P$. 
 
The above algorithm considers a large number of equivalent partial solution
matrices $R$, and is impractical for all but very small orders.  We can
obtain a vastly more efficient algorithm by imposing an ordering constraint
on the $+1$'s and $-1$'s in row-vectors.  
We do this by defining the concept of ``framings'' and a new set of associated
variables called ``frame variables'' which we use in Step~\ref{it:lin} instead
of $r_{k}$.  We first define these terms, and then give an improved
version of the algorithm.
 
At each level $k$, we create a partition of the indices $\{1,\ldots,n\}$ into
\emph{frames}, where a frame is a nonempty contiguous set of indices; 
and the collection of frames is called a \emph{framing}.
A framing of size $m$ is defined by a \emph{frame-widths vector}
$w = (w_1,\ldots,w_m) \in \pos^m$, $\sum w_i = n$, where
$w_i$ is the size of the $i$-th frame.
At level $1$ the framing consists of a single frame $\{1,\ldots,n\}$ with
frame-width vector $w = (n)$.  At each subsequent level, the framing is a
{refinement} of the framing from the previous level.  We use the framing at
level $k-1$ to define the frame variables that we use at level $k$ in the
following algorithm. The frame variable
$x_i$ gives the number of $+1$ entries in the $k$-th row of $R$,
considering only the column indices given by the $i$-th frame.
Thus, the number of $-1$ entries is $w_i - x_i$ and the sum of the
entries is $2x_i - w_i$ (see equation~(\ref{eq:lineq})).

\newpage
\subsubsection*{Algorithm using the single-Gram constraint, version 2}

\begin{enumerate}
\item Initialize the \emph{level} $k=1$, 
the \emph{frame-size} $m = 1$ and the
\emph{frame-width-vector} $w=(w_1)=(n)$.  
Set $q_1=(+1)$ and $Q=q_1$.
[In the course of the algorithm, $q_i$ is a column vector of size
$k-1$ or $k$, and $Q$ is a matrix whose columns depend on the $q_i$.
Also, $w$ may be thought of as a vector of weights 
corresponding to the columns of $Q$.]
\item Call {\tt search}$(k)$.
\item Output ``no solution'' and halt.
\item {\tt search}$(k)$:	\label{it2:lin}
If $k=n$, output the solution $Q$ and halt [here $m=n$]. 
Otherwise increment $k$. 

Define integer variables $x_1, x_2, \ldots, x_m$.  
Find all solutions
to the following integer programming problem:
\begin{equation}					\label{eq:lineq}
Q 
\left [
\begin{matrix}
2x_1 - w_1\\2x_2-w_2\\\vdots \\2x_m-w_m
\end{matrix}
\right ]
=
\left [
\begin{matrix}
g_{1,k}\\g_{2,k}\\ \vdots \\ g_{k-1, k}
\end{matrix}
\right ]
\end{equation}
subject to
\begin{equation}
0 \le x_i \le w_i \;\;\text{for}\;\; 1 \le i \le m.	\label{eq:bounds}
\end{equation}

For each solution, update $w$ and $Q$ as follows:

\begin{enumerate}
\item Let $w:= (x_1, w_1-x_1,\; x_2, w_2-x_2,\; \ldots,\; x_m, w_m-x_m)$.
\item Recall that $Q=(q_1, q_2, ..., q_m)$ is a matrix of column vectors $q_i$,
       each of length $k-1$.  Update $Q$ to a $k \times 2m$ matrix as follows:

 $Q:=
\left[
\begin{matrix}
q_1 & q_1& q_2 & q_2 & \cdots & q_m & q_m\\
+1 & -1 & +1 & -1 & \cdots & +1 & -1\\
\end{matrix}
\right].
$

\item Compress $Q$ by removing all columns which correspond to zeros in $w$.
\item Compress $w$ by removing all zero entries.
\item Set $m := \text{length}(w)$.
\end{enumerate}
Call {\tt search}$(k)$ recursively to search the relevant subtree.
When all solutions have been processed, return to the caller
(i.e.~backtrack).
\end{enumerate}

\noindent In procedure {\tt search}$(k)$ of version~2 we use
Gaussian elimination with column pivoting in order
to find a $(k-1)\times (k-1)$ nonsingular minor of $Q$
(this is always possible, since the Gram matrix $G$ is positive definite). 
We then solve for the
corresponding $k-1$ \emph{basic} variables in terms of the remaining
$m-k+1$ \emph{non-basic} variables. The non-basic variables are chosen
exhaustively as integers in the appropriate intervals given
by~(\ref{eq:bounds}); 
the basic variables are then determined uniquely (as real
numbers).  If the non-basic variables are not in $\integer$
or violate the bounds~(\ref{eq:bounds}), the solution is discarded.
It is preferable
to choose as non-basic variables the variables with the smallest
upper bounds~$w_i$, provided that the resulting $(k-1)\times (k-1)$ 
minor is nonsingular. 
A heuristic for accomplishing this is to weight
the columns in proportion to the bounds $w_i$ before performing the Gaussian
elimination.

We illustrate an iteration of the algorithm with an example. 
Consider the case $n=7$ and the candidate Gram matrix
(here and elsewhere we may abbreviate ``$-1$'' by ``$-$''):
\[G = \left[\begin{matrix}
7& 3&-&-&-&-&-\\
 3& 7& -& -& -& -& -\\
-& -& 7& 3& -& -& -\\
-& -& 3& 7& -& -& -\\
-& -& -& -& 7& 3& -\\
-& -& -& -& 3& 7& -\\
-& -& -& -& -& -& 7\\
\end{matrix}\right].\]

Suppose we are at level $k=3$ in the search.  At this stage 
the search tree has not branched yet, so there is just one matrix $Q$:
\ben
Q =
\left [
\begin{matrix}
1 & 1& 1 & 1\\
1 & 1 & - & -\\
1 & - & 1 & -
\end{matrix}
\right ]
\een
Associated with $Q$ is the frame-widths vector (at depth $3$) which is 
\ben
w=(2,3,1,1).
\een
We comment that, translated into the language of the algorithm in version~1, $Q$ and $w$ together correspond to a $3 \times 7$ matrix $R$:
\[
\left[\begin{matrix}
r_1\\
r_2\\
r_3\\
\end{matrix}\right] =
\left[\begin{matrix}
1 & 1 & 1 & 1 & 1 & 1 & 1\\
1 & 1 & 1 & 1 & 1 & - & -\\
1 & 1 & - & - & - & 1 & -\\
\end{matrix}\right].
\]
To find the next row of $Q$, we define $4$ $(=m=|w|)$ new variables $x_1,
x_2, x_3, x_4$.  The interpretation is that $x_1$ represents the number of
``$+1$''s in the first $w_1=2$ entries of row 4, $x_2$ represents the number
of ``$+1$''s in the next $w_2=3$ entries of row 4, as so on.
We use the constraints imposed 
by $g_{1,4}, g_{2,4}, g_{3,4}$, giving the linear system
\begin{equation*}
\left[
\begin{matrix}
1 & 1 & 1 & 1\\
1 & 1 & - & -\\
1 & - & 1 & -\\
\end{matrix}
\right]
\left[
\begin{matrix}
2x_1 - 2\\
2x_2 - 3\\
2x_3 - 1\\
2x_4 - 1\\
\end{matrix}
\right] =
\left[
\begin{matrix}
-\\ - \\ 3\\
\end{matrix}
\right].
\end{equation*}

The two integer
solutions which satisfy this system as well as the bounds
$0 \le x_1 \le 2$, $0 \le x_2 \le 3$, $0 \le x_3 \le 1$, $0 \le x_4 \le 1$ 
given by equation~(\ref{eq:bounds}) are
$(x_1,x_2,x_3,x_4) = (1,1,1,0)$ and $(2,0,0,1)$.
These generate two children in the search tree, with $Q$ and $w$ as follows:
\ben
Q =
\left [
\begin{matrix}
1 & 1 & 1 & 1 & 1 & 1\\
1 & 1 & 1 & 1 & - & -\\
1 & 1 & - & - & 1 & -\\
1 & - & 1 & - & 1 & -
\end{matrix}
\right ]
\text{ with $w=(1,1,1,2,1,1)$;}
\een
and
\ben
Q =
\left [
\begin{matrix}
1 & 1 & 1 & 1\\
1 & 1 & - & -\\
1 & - & 1 & -\\
1 & - & - & 1
\end{matrix}
\right ]
\text{ with $w=(2,3,1,1)$.}
\een
The first $Q$ leads to a solution; the second does not.

\subsection{Decomposition using Gram-pair constraints}
\label{subsec:nonlinear}

The algorithm (version 2) outlined in \S\ref{subsec:linear}, 
using only the single-Gram constraint, 
quickly becomes infeasible due to the size of the
search space.  
It can be improved by noting that, since the list $\cal L$
is complete, it must include a matrix $H$ Gram-equivalent to $R^TR$.  By 
permuting columns of $R$,
we can assume that $H = R^TR$.  This relation allows us
to prune the search more efficiently than if we did not know $H$.

Recall that the characteristic polynomial of a
square matrix $A$ is the monic polynomial 
$P(\lambda) = \det(\lambda I - A)$.
Since $H = R^TG(R^T)^{-1}$, the matrices $G$ and $H$ are similar, 
so they have the same characteristic polynomial.

Thus, the refined strategy is to consider each pair 
$(G,H) \in {\cal L}^2$, such that $G$ and $H$ have the same characteristic
polynomial, and try to find $R$ such that
$G = RR^T$, $H = R^TR$.  If we have considered $(G,H)$ there is 
no need to consider $(H,G)$ since this would correspond to the dual solution
$R^T$.

More precisely, consider the constraint
\begin{equation}
G^{j+1} = (RR^T)^{j+1} = R(R^TR)^jR^T = RH^jR^T. \label{eq:nonlinear}
\end{equation}
We say that the \emph{degree} of such a constraint is $j+1$, since the
elements of $G$ (not $R$) occur with degree $j+1$.
The case $j = 0$ corresponds to the single-Gram constraint considered
in \S\ref{subsec:linear}.  For $j = 1$ we get the degree 2 constraint
\[G^2 = RHR^T\]
considered in~\cite{Orrick05}.  
The use of Gram-pair constraints with $j>1$ is a new element of our algorithm.  In principle, we could
get different constraints for $j = 0, 1, \ldots, n-1$. For $j \ge n$ the
Cayley-Hamilton theorem implies that we get nothing new.

To apply Gram-pair constraints for pruning 
when only the first $k-1$ rows of $R$ are known, 
partition the matrices appearing in~(\ref{eq:nonlinear}) 
into corresponding blocks:
\[
R = \left[\begin{matrix} R_1\\R_2 \end{matrix}\right],\;\;
G^{j+1} = \left[\begin{matrix} G_{1,1}{(j+1)} & G_{1,2}{(j+1)}\\
            G_{2,1}{(j+1)} & G_{2,2}{(j+1)}\\\end{matrix}\right],\;\;
H^j = \left[\begin{matrix} H_{1,1}{(j)} & H_{1,2}{(j)}\\
            H_{2,1}{(j)} & H_{2,2}{(j)}\\ \end{matrix}\right] 
\]
say, where $R_1$ has $k-1$ (known) rows.  Then we can use the constraints
\begin{equation}
G_{1,1}{(j+1)} = R_1 H_{1,1}{(j)} R_1^T		\label{eq:block_nonlinear}
\end{equation}
since it only involves the known rows of $R$.
The matrices $G^{j+1}$ and $H^j$ need only be computed once.

Observe that $H_{1,1}{(j)}$ for $j \ge 2$ 
depends on all the entries in $H$.
This suggests that~(\ref{eq:block_nonlinear}) with $j \ge 2$ may be
more effective for pruning than the ``degree 2'' case $j = 1$.
In practice, we found that it was worthwhile to 
use~(\ref{eq:block_nonlinear}) with both $j = 1$ and $j = 2$, but 
not with $j>2$.

When using Gram-pair constraints for pruning, we can no longer
assume that the first row of $R$ is $(1, 1, \ldots, 1)$.
The algorithm (version~2) of \S\ref{subsec:linear} has to be modified so
step~1 starts with level $k=0$, $Q$ empty, 
and a frame-widths vector $w$ which
is compatible with $H$, in the sense that $H$ is invariant under
permutations of rows (and corresponding columns) within each frame.
For example, if we take $H = G$ in the example of order~$7$ above,
we can choose $w = (2, 2, 2, 1)$ as the initial frame-widths vector.

We remark that we used three variants of the decomposition
algorithm outlined in this subsection.  One variant attempts to find
a decomposition or (by failing to do so) to prove that a decomposition
of a given pair $(G,H)$ does not exist. A second variant finds all possible
decompositions, up to Hadamard equivalence. The output typically includes
many solutions that are Hadamard equivalent, so we use McKay's program
\emph{nauty}~\cite{nauty} to remove all but one representative of each 
equivalence class after transforming the problem to a graph isomorphism
problem~\cite{McKay79}.  A third variant is nondeterministic and attempts
to traverse the search tree by choosing one or more children randomly at
each node. This variant is useful in difficult cases where the deterministic
variants take too long (see \S\ref{subsec:decomp37} for an example).

\subsection{Proving indecomposability using the Hasse--Minkowski criterion}
A complementary approach to the decomposition problem, or, more properly, to proofs of indecomposability, makes use of the Hasse--Minkowski theorem on rational equivalence of quadratic forms.  The use of this theorem has a long history in design theory, originating with its use by Bruck and Ryser in their proof of their nonexistence result for certain finite projective planes~\cite{BrRy49}.  Tamura recently applied the theorem to the question of decomposability of candidate Gram matrices with block structure~\cite{Tamura06}.
\begin{theorem}
Let $A$ and $B$ be symmetric, nonsingular rational matrices of the same dimension.  Then there exists a rational matrix $R$ such that $B=RAR^T$ if and only if
\begin{enumerate}
\item $\det A$/$\det B$ is a rational square, and
\item the $p$-signatures  of $A$ and $B$ agree for all primes $p$ and for $p=-1$.
\end{enumerate}
\end{theorem}
The criterion is implemented by finding rational matrices $U$ and $V$ such that $UAU^T$ and $VBV^T$ are diagonal---this can always be done---and then by comparing the $p$-signatures of the resulting matrices for all primes dividing any of the diagonal elements.  The $p$-signature of a diagonal form is defined in~\cite[Chapter 15, \S5.1]{ConSlo98}.

The application of this theorem is as follows: there is no decomposition of the form $RR^T=G$, $R^TR=H$ if there is a $j\ge0$ for which $G^{j+1}$ fails to be rationally equivalent to $H^j$ or for which $H^{j+1}$ fails to be rationally equivalent to $G^j$.  As was the case in the application of the Gram-pair constraint in back-tracking search, we need only check the criterion for $j<n$.  

The Hasse--Minkowski criterion is sometimes a competitive alternative to the backtracking algorithm in ruling out the existence of a decomposition.  On certain Gram matrix pairs for which the back-tracking search ruled out a decomposition only after exploring the search tree to great depth, the Hasse--Minkowski criterion ruled out any decomposition with a relatively fast computation.  In most cases, however, back-tracking search is the much faster method, especially when rational equivalence fails only for large $j$, in which case the large-integer arithmetic needed to implement the Hasse--Minkowski criterion can become prohibitively expensive.  Furthermore, in a small number of cases, the Hasse--Minkowski theorem fails entirely to rule out a decomposition where backtracking search succeeds.  It is surprising that this occurs relatively infrequently, as the existence of a decomposition with $R$ rational would appear to be a far milder constraint than the existence of a decomposition with $R$ a $\{+1,-1\}$ matrix.

\section{The Maximal Determinant for Order 19}	\label{sec:order19}

For order 19, known designs due to Smith~\cite{Smith88}, Cohn~\cite{Cohn00}
and Orrick and Solomon~\cite{OS-19-R3} are D-optimal, as we now show.

\begin{samepage}
\begin{theorem} \label{thm:19}
The maximal determinant of $\{+1, -1\}$ order $19$ matrices  is
\begin{equation} \label{eqn:maxdet19}
2^{30} \times 7^2 \times 17 =
833 \times 4^6 \times 2^{18}.
\end{equation}
There are precisely three corresponding equivalence classes 
of saturated D-optimal designs with representatives $R_1, R_2$ and $R_3$
indicated in Figure~\ref{fig:designs19}.  There are two corresponding (Gram equivalence classes of) Gram matrices,
$G_1 = R_1R_1^T = R_1^T R_1$ and
$G_2 = R_2R_2^T = R_2^T R_2 = R_3R_3^T = R_3^T R_3$ -- see 
Figure~\ref{fig:gram19}. 
\end{theorem}
\end{samepage}

\begin{proof}
A computational proof of Theorem~\ref{thm:19} is described in
\S\S\ref{subsec:gram19}--\ref{subsec:decomp19}.
\end{proof}

\begin{remark} The maximal determinant given by {\rm(\ref{eqn:maxdet19})} 
is smaller by a factor $17/\sqrt{304} \approx 0.975$
than the Ehlich bound~{\rm(\ref{eq:e_bound})}
for $n=19$.
\end{remark}

\subsection{Candidate Gram-finding for order 19} \label{subsec:gram19}

In the algorithm described in \S\ref{sec:gram-finding}
we used $d_{\min} = 833 \times 4^6 \times
2^{18}$ since $\{+1,-1\}$ matrices with this determinant were known to
exist. Our candidate Gram-finding program took 826 hours\footnote{
Computer times mentioned here and below are for a single 2.3GHz Opteron
processor. In cases where the search could easily be parallelised,
we sometimes used several processors running in parallel. Our 
candidate Gram-finding program actually took 188 hours using several
processors, each operating on part of the search tree.
Our programs were written in C and used the GMP package 
to perform multiple-precision arithmetic.}
to find nine equivalence classes of candidate Gram matrices and to rule out
any others. 

For the nine candidate Gram matrices $G$, the values of
$\sqrt{\det(G)}/2^{30}$ were $840$ (five times),
$836.0625$ (once), $836$ (once), and $833$ (twice).
The matrices are available from the website~\cite{www37}.

\begin{center}
\begin{figure}[htbp]
\input "G1-19.tex"
\input "G2-19.tex"
\caption{Optimal Gram matrices for $n=19$.  
	 Here ``$-$" stands for ``$-1$".}
\label{fig:gram19}
\end{figure}
\end{center}

\subsection{Decomposition for order 19} \label{subsec:decomp19}

Our decomposition program found that seven of the candidate Gram
matrices were indecomposable, but the last two decomposed (as expected).
The running time was only 0.85 sec.  
Nevertheless, it would be extremely
tedious to attempt to replicate the search by hand, since it involves
visiting about $1400$ nodes in the search trees.

The nine matrices have distinct characteristic polynomials, so we only
had to consider the case $G = RR^T = H = R^TR$. Only the
two candidate Gram matrices with smallest determinant were decomposable, and
these decomposed in three ways, giving three Hadamard classes of
designs (maxdet matrices) of order $19$. See Figure~\ref{fig:gram19} for
the Gram matrices (note that they differ only in the first row and column), 
and Figure~\ref{fig:designs19} for two of the three designs.
The third design $R_3$ can be obtained from $R_2$ by a switching
operation, as indicated in Figure~\ref{fig:designs19}.

A variant of our decomposition program exhaustively searches for
all possible decompositions (up to equivalence) of a given pair $(G,H)$.
Running this program on $(G_1,G_1)$ gave $110592$ matrices in 36 seconds.
Using McKay's program \emph{nauty}~\cite{McKay79,nauty},
we verified that they were all equivalent to $R_1$.  
Similarly, on $(G_2,G_2)$ we obtained $3456$ matrices in 3 seconds, 
and \emph{nauty}
verified that $1728$ were equivalent to $R_2$, and the remaining
$1728$ were equivalent to $R_3$.  Thus, there are precisely three
inequivalent designs with maximal determinant.

\begin{center}
\begin{figure}[htbp]
\input "R1-19.tex"
\input "R23-19.tex"
\caption{Two inequivalent saturated D-optimal designs of order $19$.  Here ``$-$" stands for ``$-1$" and ``$+$" stands for ``$+1$".
A third inequivalent design $R_3$ is the same as $R_2$ except 
that the circled entries 
have their signs reversed (this is an example of ``switching'', 
see Orrick~\cite{Orrick08a}).}
\label{fig:designs19}
\end{figure}
\end{center}

\section{The Maximal Determinant for Order 37}	\label{sec:order37}

The case of order 37 was handled in much the same way as order 19. Although
37 is much larger than 19, we have $37 \equiv 1 \bmod\ 4$, and typically
the cases $1 \bmod\ 4$ are easier than the cases $3 \bmod 4$ (as one can
see from the summary at~\cite{maxdet}). This is partly because 
Theorem~\ref{thm:enhancedKM} gives a sharper bound when $n \equiv 1 \bmod 4$.
 
We established that, for order $37$,
a known design, found previously by Orrick and Solomon~\cite{OrSo07},
is D-optimal. The design is not unique, but the corresponding 
Gram matrix is (up to equivalence).

\begin{theorem} \label{thm:37}
The maximal determinant of $\{+1, -1\}$ order $37$ matrices  is
\begin{equation} \label{eqn:maxdet37}
72 \times 9^{17} \times 2^{36} = 2^{39} \times 3^{36}.
\end{equation}
A representative $R$ of one equivalence class 
of saturated D-optimal designs is
indicated in Figure~\ref{fig:design37}.
The corresponding Gram matrix is
$G = RR^T = R^T R$ as shown in Figure~\ref{fig:gram37}.
Moreover, $G$ is unique, up to equivalence.
\end{theorem}

\begin{remark} The maximal determinant given by {\rm(\ref{eqn:maxdet37})} 
is smaller by a factor $8/\sqrt{73} \approx 0.936$
than the Ehlich-Barba bound~{\rm(\ref{eq:eb_bound})}
for $n=37$.
\end{remark}

\begin{center}
\begin{figure}[hbp]
\input "G-37.tex"
\caption{The optimal Gram matrix for $n=37$.
	 All omitted entries are $1$.}
\label{fig:gram37}
\end{figure}
\end{center}

The matrix $R$ in Figure~\ref{fig:design37} was constructed by Orrick
and Solomon from a doubly 3-normalized Hadamard matrix of order~36.
There are at least 78 (and probably many more) inequivalent designs,
as we discuss below.

\subsection{Candidate Gram matrices for order 37}

Our backtracking program with bound $d_{\min} = 2^{39}3^{36}$
(93.6\% of the Ehlich-Barba bound) took 
77 hours 
to find 807 candidate Gram matrices. These had 284 distinct determinants
$\Delta^2$ in the range 
$\Delta/({2^{39}3^{32}}) \in [81,85].$
The candidate Gram matrices are available from the website~\cite{www37}.

\subsection{Decomposition for order 37}		\label{subsec:decomp37}

We applied our decomposition program to all pairs $(G,H)$ of candidate Gram
matrices where $G$ and $H$ had the same characteristic polynomial.
There were $489$ different characteristic polynomials,
and $1528$ pairs $(G,H)$ to consider.
The decomposition algorithm took 
$257$ seconds 
to show that $806$ of the
candidate Gram matrices did not decompose (in no case could more than two
rows of $R$ be constructed).

For the remaining candidate, which was in fact equivalent to the Gram matrix
$G$ shown in Figure~\ref{fig:gram37}, the program was stopped after running
for $147$ hours and exploring about $1.7 \times 10^8$ nodes (reaching level 26
of the tree). 

A variant of our decomposition program uses a randomised search~-- at each
node of the tree being searched, we choose to explore one (or sometimes two) 
children selected uniformly at random.
Using this randomised search program we can decompose $G$, 
in fact we have now found 39 solutions. Finding one solution
takes on average about 125 hours. By also considering duals, 
we get 78 solutions.
Some (but not all) of these can be obtained from a Hadamard
matrix of order $36$, in the same way as 
the matrix $R$ of Figure~\ref{fig:design37}.
Since all $78$ solutions are inequivalent,
we expect that many more inequivalent solutions exist.
The known solutions are available from the website~\cite{www37}.

\begin{center}
\begin{figure}[htp]
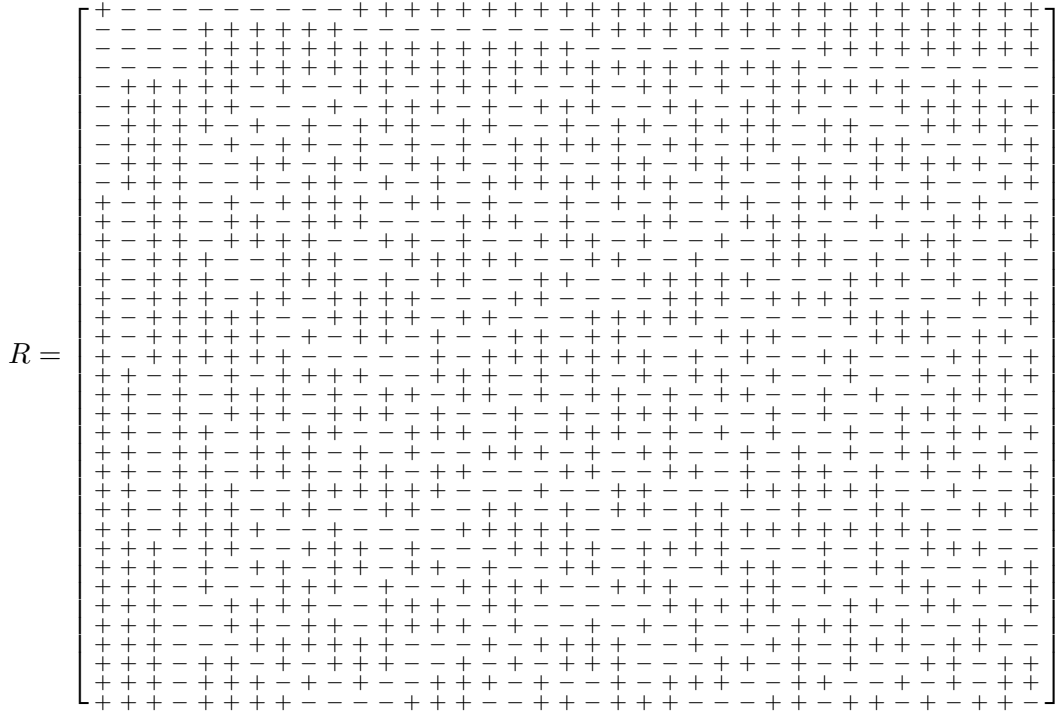

\input "R-37.tex"
\caption{A saturated D-optimal design of order $37$,
constructed by Orrick and Solomon.  
Here ``$-$" stands for ``$-1$" and ``$+$" stands for ``$+1$".}
\label{fig:design37}
\end{figure}
\end{center}

\section{Improved Bounds for Various Orders} \label{sec:other_bounds}

Recall that $d_n$ is the maximal determinant for order $n$, divided by
the known factor $2^{n-1}$.
Table~1 summarises the best known upper and lower bounds on $d_n$
for orders $n = 19, 29, 33, 37, 45, 49, 53, 57$
(we omit \hbox{$n = 25, 41, 61$} 
because for these orders
the Ehlich-Barba bound~(\ref{eq:eb_bound})
is attained). 
The figures in parentheses are the ratios of
the entries to the 
Ehlich-Barba bound
(for $n \equiv 1 \bmod 4$)
or the Ehlich bound 
(for $n \equiv 3 \bmod 4$),
rounded to three decimals.

For $n = 19$ and $n = 37$, the upper and lower bounds are equal, and thus
optimal ($d_n = u = \ell$ in these cases).
In the other cases the upper bounds are unattainable, so $d_n \in [\ell, u)$.
In all cases the upper bounds are new, and for $n=45$ the lower bound is
new (the previous best lower bound was $83 \times 11^{21}$).

The last column gives the number of equivalence classes of candidate
Gram matrices $G$ with $\det(G) \ge (2^{n-1}u)^2$. 

\begin{center}
\begin{tabular}{|c|cc|cc|c|}
\hline
{\rm order} $n$ & 
\multicolumn{2}{|c|}{lower bound $\ell$} & 
\multicolumn{2}{|c|}{upper bound $u$} &
Gram count\\[2pt]
\hline
&&&&&\\[-10pt]
$19$ & $833 \times 4^6$ & $(0.975)$ 
	& {$833 \times 4^6$} & $(0.975)$ & 9\\[1pt]
$29$ & $320 \times 7^{12}$ & $(0.865)$
	& {$329 \times 7^{12}$} & $(0.889)$ & 9587 \\[1pt]
$33$ & $441 \times 8^{14}$ & $(0.855)$
	& {$470 \times 8^{14}$} & $(0.911)$ & 13670 \\[1pt]
$37$ & $8 \times 9^{18}$ & $(0.936)$ 
	& {$8 \times 9^{18}$} & $(0.936)$ & 807 \\[1pt]
$45$ & $89 \times 11^{21}$ & $(0.858)$ 
	& $99 \times 11^{21}$ & $(0.953)$ & 1495 \\[1pt]
$49$ & $96 \times 12^{23}$ & $(0.812)$ 
	& $114 \times 12^{23}$ & $(0.965)$ & 168 \\[1pt]
$53$ & $105 \times 13^{25}$ & $(0.788)$ 
	& $129 \times 13^{25}$ & $(0.968)$ & 220\\[1pt]
$57$ & $133 \times 14^{27}$ & $(0.894)$	
	& $145 \times 14^{27}$ & $(0.974)$ & 128 \\[1pt]
\hline
\end{tabular}
\end{center}
\begin{center}
Table 1: Bounds on the (scaled) maximal determinant $d_n$
\end{center}

\subsection{Discussion}

Let $k = \lfloor n/4 \rfloor$. From the summary 
at~\cite{maxdet} we observe that,
in the cases $k \le 3$ 
where $d_{4k+3}$ is known precisely, $d_{4k+3}$ 
is divisible by $k^{2k-1}$. 
However, our result for $n=19$ shows that this pattern does
not continue, for $d_{19}$ is not divisible by $4^7$.

In all cases where $d_{4k+1}$ is known precisely
($k \le 6$ and $k = 10, 15, 28, \ldots$), $d_{4k+1}$
is divisible by $k^{2k-1}$.  This is easily seen to be true
if the 
Ehlich-Barba bound (\ref{eq:eb_bound}) is attainable 
($d_{4k+1}$ is divisible by $k^{2k}$ in such cases),
but it is also true for $k = 2, 4, 5$, where the 
Ehlich-Barba bound is not attainable.

For $n=29$, we ruled out $9587$ candidate Gram matrices to
show that $d_{29} < 329 \times 7^{12}$.
If $d_{29}$ is divisible by $7^{13}$, then we must
have $d_{29} = 322 \times 7^{12} = 46 \times 7^{13}$, 
since this is the only multiple of $7^{13}$ in the allowable interval
$[320\times 7^{12},\, 329 \times 7^{12})$.
However, all attempts to
construct an example of order $29$
with $|\det|/2^{28} > 320\times 7^{12}$, using hill-climbing or constructions
based on Hadamard matrices of order 28, have failed.
Thus, the plausible conjecture that $d_{4k+1}$ is divisible by $k^{2k-1}$
may well be false.
An attempt to reduce the upper bound $u$ to $322 \times 7^{12}$ 
is underway but may
not be feasible with our current resources~-- so far we have
generated $16683$ candidate Gram matrices (taking about two processor-years)
but estimate that there are 
about $220000$ in all. 

For $n=33$ we ruled out 
$13670$ candidate Gram matrices to establish an upper 
bound of $u = 470 \times 8^{14}$. It is unlikely that we can reduce
$u$ much further without improvements in the candidate Gram-finding
program, since it took about two processor-years to generate
the candidates for $u = 470$, though only about six hours to show that
none of them decompose.

For $n=45$, the new lower bound of $89 \times 11^{21}$ was
established by a construction using 
a doubly 3-normalized Hadamard matrix of order~$44$.
Details will appear elsewhere.  

\section{The Spectrum for Order 13}	\label{sec:spectrum}

The \emph{spectrum} $S_n$ of the determinant function for $\{+1,-1\}$ 
matrices is defined to
be the set of values taken by $|\det(R_n)| / 2^{n-1}$ as $R_n$ 
ranges over all $n\times n$ $\{+1,-1\}$ matrices. For $2 \le n \le 7$, 
the spectrum includes all integers
between $0$ and $d_n$. 
The spectrum for
$n=8$ was first computed by 
Metropolis, Stein, and Wells~\cite{Metropolis69}, who found that
gaps occur, 
in fact $S_8 = \{0,1,\ldots,18,20,24,32\}$.
A (non-computer-based) proof of the existence of gaps was later
given by Craigen~\cite{Craigen90}. 
At present, the spectrum is known for $n \le 11$ and (given here for the
first time) for $n=13$. 
The results for $n=9$ are due to \v{Z}ivkovi\'{c}~\cite{Zivkovic06}
(and, independently, Charalambides~\cite{Charalambides05}),
those for $n=10$ are due to \v{Z}ivkovi\'{c}~\cite{Zivkovic06},
and those for $n=11$ are due to Orrick~\cite{Orrick05}.
The spectra for $n \le 11$, and 
conjectured spectra for $n=12$ and $14 \le n \le 17$, 
may be found
at~\cite{spectrum}. Here we give only the spectrum for $n=13$,
using the notation $a{\dotdot}b$ as a shorthand to represent 
the interval $\{x \in \integer: a \le x \le b\}$.
\begin{theorem}
\noindent
The spectrum for order $13$ is\\
$S_{13} = \{
0{\dotdot}2172,
2174{\dotdot}2185,
2187{\dotdot}2196,
2199{\dotdot}2202,
2205,
2208,
2210, 
2211,\linebreak
2214{\dotdot}2218,
2220{\dotdot}2226,
2228, 
2229, 
2230,
2232, 
2233,
2235,
2238,
2240, 
2241,\linebreak
2243{\dotdot}2245,
2247, 
2248,
2250,
2253,
2256,
2258{\dotdot}2260,
2262,
2264, 
2265,
2267,\linebreak 
2268,
2271, 
2272,
2274,
2277,
2280,
2283,
2286,
2288,
2292,
2295, 
2296,
2304,\linebreak
2307,
2312, 
2313,
2316,
2319,
2320,
2322,
2325,
2328,
2331,
2334,
2336,
2340,\linebreak
2343, 
2344,
2349,
2352,
2355,
2360, 
2361,
2367, 
2368,
2370,
2373,
2376,
2385,\linebreak
2394,
2400,
2403,
2406,
2421,
2430,
2432,
2439,
2457,
2472,
2484,
2496,
2511,\linebreak
2520,
2538,
2560,
2583,
2592,
2619,
2646,
2673,
2835,
2916,
3159,
3645\}$.
\end{theorem}
\begin{proof}
The proof is computational.  Using a heuristic algorithm
described in~\cite[pg.~34]{Orrick-spectrum-talk}, we found examples of
order~$13$ matrices with all $2173$ determinants 
$0, 1\times 2^{12}, 2 \times 2^{12}, \ldots, 2172\times 2^{12}$.
The first ``gap'' was at $2173\times 2^{12}$.  
We ran the Gram-finding program
of~\S\ref{sec:gram-finding} with lower bound $d_{\min} = 2173\times 2^{12}$.
It produced $8321$ candidate Gram matrices in $73$ minutes.
We then ran the decomposition program of~\S\ref{sec:decomp} which found 
(in $48$ seconds) that $1643$ 
of the candidate Gram matrices decomposed, 
giving $130$ distinct determinants in the range $[2174,3645]$.
These are listed in the statement of the theorem.
\end{proof}

\appendix
\section*{Appendix: Proof of new bound~(\ref{eqn:evensharperbound}) in
Theorem~\ref{thm:enhancedKM}}
The proof is a generalisation Ehlich's proof~\cite{Ehlich64a} of the
bound~(\ref{eq:e_bound}) on the maximal determinant, which applies in the
case $n\equiv3\pmod{4}$.  Ehlich's proof shows the following.
\begin{enumerate}
\item[(1)] The candidate principal minor of maximal determinant has 
non-diagonal elements equal to either $-1$ or $3$.
\item[(2)] It is a {\em block matrix}, which means that the 
non-diagonal 3s occur in square blocks along the diagonal. 
\item[(3)] In the case that the candidate principal minor is a 
candidate Gram matrix, that is, its size is $n$, the number of blocks is $s$, 
with $u$ blocks of size $\lfloor n/s\rfloor$ and $v$ blocks of 
size $\lfloor n/s\rfloor+1$, where $s$, $u$, and $v$ are defined 
following~(\ref{eq:e_bound}).
\end{enumerate}
We generalise this to the case where the candidate principal minor contains 
a fixed principal submatrix $M_r$.  If such a candidate principal minor of 
maximal determinant is written as
\begin{equation*}
\begin{bmatrix}M_r & B\\B^T & A\end{bmatrix},
\end{equation*}
then our result, assuming the hypothesis $\det M_r>(n-3)\det M_{r-1}$, 
is that $A$ satisfies properties (1) and (2) above, and the submatrices 
of $B$ corresponding to blocks of $A$ consist of repeated columns.  
The generalisation of (3) depends on $M_r$ and is not unique in general.

{From} now on, we take $n\equiv3\pmod{4}$ and $n>3$.  Define
\begin{align}
\mathfrak{C}_m=&\left\{C_m|C_m=(c_{ij}),c_{ij}=c_{ji},C_m\text{ is pos.\ def.},c_{ii}=n,\right.\notag\\
&\ \left.c_{ij}\equiv n\pmod{4},i,j=1\ldots m\right\}.
\end{align}
and
\begin{equation}
\mathfrak{E}_m=\{E_m|\text{$E_m\in \mathfrak{C}_m$ and the 
leading $r\times r$ submatrix of $E_m$ is $M_r$}\}.
\end{equation}
Define $C_m^*$ and $E_m^*$ (which  may not be unique) by the conditions
\begin{align*}
\det C_m^*&=\max\{\det C_m|C_m\in\mathfrak{C}_m\},\\
\det E_m^*&=\max\{\det E_m|E_m\in\mathfrak{E}_m\}.
\end{align*}

The first thing Ehlich proves (Theorem 2.1) is that 
$\det C_m^*>(n-3)\det C_{m-1}^*$ for $2\le m\le n$.  
To explain the use of this theorem, we first introduce a notation.  
If $C_m\in\mathfrak{C}_m$ then define ${\tilde C}_m$ be the matrix 
that results from replacing the last diagonal element of $C_m$ by 3, 
i.e. ${\tilde C}_m=({\tilde c}_{ij})$ where
\begin{align*}
\tilde c_{ij}=\begin{cases}
3 & \text{if $i=j=m$}\\
c_{ij} & \text{otherwise.}
\end{cases}
\end{align*}
Expanding $\det C_m^*$ by minors on its last row, we find
\begin{equation}
\det C_m^*-\det\tilde C_m^*=(n-3)\det C_{m-1}\le(n-3)\det C_{m-1}^*
\end{equation}
where $C_{m-1}$ is the leading $(m-1)\times(m-1)$ submatrix of $C_m^*$. 
(Note that $\tilde C_m^*$ is the result of applying the tilde operation to
$C_m^*$.) The theorem then implies that $\det\tilde C_m^*>0$ and therefore
that $\tilde C_m^*$ is positive definite.  This becomes important in later
proofs when evaluating determinants that arise as the result of column
operations.

For the generalization to our case, we appear to need the extra condition
$\det\tilde M_r>0$.  This cannot be expected to hold in general.  Consider
for example,
\begin{equation}
M_2=\begin{bmatrix}11 & 7\\7 & 11\end{bmatrix}
\end{equation}
for which $\det\tilde M_2<0$.  For now, we leave it as a question for
empirical study whether the condition holds often enough in practical
searches for the following considerations to be useful.

\begin{theorem}
Let $1\le r\le m$, let $M_r\in\mathfrak{C}_r$, and let $\det\tilde M_r>0$. 
Then the set of elements $E_m\in\mathfrak{E}_m$ for which $\det\tilde E_m>0$
is non-empty.
\end{theorem}
\begin{proof}
We prove the theorem by induction on $m$.  For the base case, $m=r$, the
matrix $M_r\in\mathfrak{E}_r$ satisfies $\det\tilde{M}_r>0$ by assumption. 
Now assume that the theorem holds for all $\mathfrak{E}_k$ with $r\le k\le m$. 
We want to show that it holds for $\mathfrak{E}_{m+1}$.  Let $E_m$ be
an element of $\mathfrak{E}_m$ for which $\det\tilde E_m>0$, and write this
element as
\begin{equation}
E_m=\begin{bmatrix}C_{m-1} & \gamma\\\gamma^T & n\end{bmatrix}.
\end{equation}
Form the matrix
\begin{equation}
E_{m+1}=\begin{bmatrix}C_{m-1} & \gamma & \gamma\\\gamma^T & n & 3\\\gamma^T & 3 & n\end{bmatrix}.
\end{equation}
Subtracting column $m$ from column $m+1$ and expanding by minors on column
$m+1$ we find that $\det E_{m+1}=(n-3)\det E_m+(n-3)\det\tilde E_m$.  Both
terms are positive, which means that $\det E_{m+1}$ is positive, and
therefore that $E_{m+1}$ is positive definite and hence an element of
$\mathfrak{E}_{m+1}$.  By a similar computation $\det\tilde
E_{m+1}=(n-3)\det\tilde E_m>0$.  Therefore $E_{m+1}$ is a suitable element.
\end{proof}

\begin{theorem}
Let $1\le r<m$, let $M_r\in\mathfrak{C}_r$, and let $\det\tilde M_r>0$. Then
\begin{equation*}
\det E_m^*>(n-3)\det E_{m-1}^*.
\end{equation*}
\end{theorem}
\begin{proof}
When $m=r+1$ we write
\begin{equation}
M_r=\begin{bmatrix}C_{r-1} & \gamma\\\gamma^T & n\end{bmatrix}
\end{equation}
and define
\begin{equation}
E_{r+1}=\begin{bmatrix}C_{r-1} & \gamma & \gamma\\\gamma^T & n & 3\\\gamma^T & 3 & n\end{bmatrix}.
\end{equation}
{From} the proof of the previous theorem we know that
$E_{r+1}\in\mathfrak{E}_{r+1}$ and $\det\tilde E_{r+1}>0$.  Now $\det
E_{r+1}^*\ge\det E_{r+1}=(n-3)\det M_r+(n-3)\det\tilde M_r>(n-3)\det
M_r=(n-3)\det E_r^*$.

We now proceed by induction.  Assume that $\det E_m^*>(n-3)\det E_{m-1}^*$. 
Write
\begin{equation}
E_m^*=\begin{bmatrix}E_{m-1} & \gamma\\\gamma^T & n\end{bmatrix}
\end{equation}
and define
\begin{equation}
E_{m+1}=\begin{bmatrix}E_{m-1} & \gamma & \gamma\\\gamma^T & n & 3\\\gamma^T
& 3 & n\end{bmatrix}.
\end{equation}
Now $\det E_{m+1}=(n-3)\det E_m^*+(n-3)\det\tilde E_m^*$.  Note that $\det
E_m^*=(n-3)\det E_{m-1}+\det\tilde E_m^*\le(n-3)\det E_{m-1}^*+\det\tilde
E_m^*$.  From the induction hypothesis, it follows that $\det\tilde
E_m^*>0$, and so $\det E_{m+1}>(n-3)\det E_m^*$.
\end{proof}

This proof contains the proof of an important corollary:

\begin{corollary}\label{coro:tildePosDef}
Let $1\le r\le m$, let $M_r\in\mathfrak{C}_r$, and let $\det\tilde M_r>0$. 
Then $\det\tilde E_m^*>0$.
\end{corollary}

We now generalize Ehlich's Theorem 2.2.  Again we need the assumption that
$\det\tilde M_r>0$.

\begin{theorem}\label{theo:minusOneAndThree}
Let $1\le r\le m$, let $M_r\in\mathfrak{C}_r$, and let $\det\tilde M_r>0$. 
Write
\begin{equation*}
E_m^*=\begin{bmatrix}M_r & B\\B^T & A\end{bmatrix},
\end{equation*}
where $A=(a_{ij})$ and $B=(b_{ij})$ satisfy the conditions in the definition
of $\mathfrak{E}_m$.  Then for $i\ne j$ we have $a_{ij}=-1$ or $3$.
\end{theorem}
\begin{proof}
When $m=r$ or $r+1$ the statement is vacuously true.  So we assume $m\ge r+2$.  Suppose there is an element $a_{ij}=c\ne-1$ or $3$ for $i\ne j$.  Then $|c|>3$.  We may assume that the element is positioned so that $i=m-1$, $j=m$, so that we may write
\begin{equation*}
E_m^*=\begin{bmatrix}E_{m-2} & \alpha & \beta\\\alpha^T & n & c\\\beta^T & c
& n\end{bmatrix}.
\end{equation*}
By interchanging the last two rows and last two columns, if necessary, we
may assume that
\begin{equation*}
\det\begin{bmatrix}E_{m-2} & \alpha\\\alpha^T & 
n\end{bmatrix}\le\det\begin{bmatrix}E_{m-2} & 
\beta\\\beta^T & n\end{bmatrix}.
\end{equation*}
We now claim that the matrix
\begin{equation*}
E_m=\begin{bmatrix}E_{m-2} & \beta & \beta\\\beta^T & n & 3\\\beta^T & 3 &
n\end{bmatrix}
\end{equation*}
has larger determinant than $E_m^*$, a contradiction.  To establish the
claim, evaluate both determinants:
\begin{align*}
\det E_m^*&= (n-3)\det\begin{bmatrix}E_{m-2} & 
\alpha\\\alpha^T & n\end{bmatrix}+\det\tilde E_m^*\\
\det E_m&=(n-3)\det\begin{bmatrix}E_{m-2} & 
\beta\\\beta^T & n\end{bmatrix}+\det\tilde E_m
\end{align*}
By Corollary~\ref{coro:tildePosDef}, $\tilde E_m^*$ is positive definite. 
Symmetric row and column operations do not affect positive definiteness, so
we get
\renewcommand{\arraystretch}{1.2}
\begin{align*}
\det\tilde E_m^*&=\det\begin{bmatrix}E_{m-2} & \alpha & \beta\\\alpha^T & n & c\\\beta^T & c & 3\end{bmatrix}=\det\begin{bmatrix}E_{m-2} & \beta & \alpha-\frac{c}{3}\beta\\
\beta^T & 3 & 0\\
\alpha^T-\frac{c}{3}\beta^T & 0 & n-\frac{c^2}{3}\end{bmatrix}\\
&\le \left(n-\frac{c^2}{3}\right)\det\begin{bmatrix}E_{m-2} & 
\beta\\\beta^T & 3\end{bmatrix}.
\end{align*}
\renewcommand{\arraystretch}{1}
We evaluate $\det\tilde E_m$ by subtracting column $m$ from column $m-1$ and
doing expansion by minors on column $m-1$ to obtain
\begin{equation*}
\det\tilde E_m=(n-3)\det\begin{bmatrix}E_{m-2} & \beta\\\beta^T & 
3\end{bmatrix}.
\end{equation*}
Since $|c|>3$ we have $\det\tilde E_m^*<\det\tilde E_m$ and therefore $\det
E_m^*<\det E_m$.
\end{proof}

Now we want to generalize Ehlich's Theorem 2.3.  First a useful lemma about
block matrices.  (See Ehlich's paper for the formal definition of block.)
\begin{lemma}\label{lem:blockProperty}
A symmetric matrix $A=(a_{ij})$ with diagonal elements $n$ is a block matrix
if and only if its non-diagonal elements are all $-1$ or $3$ and for any
$i\ne j$ such that $a_{ij}=3$ the columns $i$ and $j$ differ only in their
$i^\text{th}$ and $j^\text{th}$ elements.
\end{lemma}
\begin{proof}
If $A$ is a block matrix, the statement is clearly true.  For the converse,
let $(i,j)$ be the position of one of the 3s of $A$.  Define $i_1=j$,
$i_2=i$, and let $\{i_2,i_3,\ldots,i_p\}$ be the set of all indices $h$ for
which $a_{hj}=3$.  Let $2\le k\le p$.  Then $a_{i_kj}=3$ means that the
$i_k^\text{th}$ element of column $j$ is 3. Let $1\le\ell\le p$, $\ell\ne
k$.  Then since columns $i_\ell$ and $j$ agree in their $i_k^\text{th}$
element we have $a_{i_ki_\ell}=3$ for all $k\ne\ell$.

For an index $h\notin\{i_1,\ldots,i_p\}$ we have $a_{hj}=-1$.  But since
columns $j$ and $i_\ell$, $1\le\ell\le p$, agree in their $h^\text{th}$
element, we have $a_{hi_\ell}=-1$ for all $1\le\ell\le p$.  Therefore the
set of indices $\{i_1,\ldots,i_p\}$ forms a block.  Hence every one of the
3s in $A$ lies in a block, and $A$ is a block matrix.
\end{proof}

Now for the generalization of Ehlich's Theorem 2.3.
\begin{theorem}\label{theo:duplCols}
Let $1\le r\le m$, let $M_r\in\mathfrak{C}_r$, and let $\det\tilde M_r>0$.  Write
\begin{equation*}
E_m^*=\begin{bmatrix}M_r & B\\B^T & A\end{bmatrix},
\end{equation*}
where $A=(a_{ij})$ and $B=(b_{ij})$ satisfy the conditions in the definition
of $\mathfrak{E}_m$.  If for some $i\ne j$, $a_{ij}=3$, then columns $i$ and
$j$ of $B$ are equal and columns $i$ and $j$ of $A$ are equal except for
their $i^\text{th}$ and $j^\text{th}$ elements.
\end{theorem}
\begin{proof}
By Theorem~\ref{theo:minusOneAndThree} we know that the non-diagonal
elements of $A$ are $-1$ or $3$.  As before, the theorem is vacuously true
if $m=r$ or $m=r+1$.  Assume $m\ge r+2$ and let $A$ have an element 3, say
in position $(m-1,m)$.  Write
\begin{equation*}
E_m^*=\begin{bmatrix}M_r & B_1 & \beta_{m-1} & \beta_m\\ B_1^T & A_1 &
\alpha_{m-1} & \alpha_m\\ \beta_{m-1}^T & \alpha_{m-1}^T & n & 3\\\beta_m^T
& \alpha_m^T & 3 & n\end{bmatrix}.
\end{equation*}
We assume, as we may (by swapping the last two rows and last two columns if
necessary), that
\begin{equation*}
\det\begin{bmatrix}M_r & B_1 & \beta_{m-1}\\ B_1^T & A_1 & \alpha_{m-1}\\ \beta_{m-1}^T & \alpha_{m-1}^T & n\end{bmatrix}\le\det\begin{bmatrix}M_r & B_1 & \beta_m\\ B_1^T & A_1 & \alpha_m\\ \beta_m^T & \alpha_m^T & n\end{bmatrix}.
\end{equation*}
Our goal is now to show that $\beta_{m-1}=\beta_m$ and
$\alpha_{m-1}=\alpha_m$. Suppose that this is not the case.  We claim that
$\det E_m>\det E_m^*$ where
\begin{equation*}
E_m=\begin{bmatrix}M_r & B_1 & \beta_m & \beta_m\\ B_1^T & A_1 & \alpha_m &
\alpha_m\\ \beta_m^T & \alpha_m^T & n & 3\\\beta_m^T & \alpha_m^T & 3 &
n\end{bmatrix}.
\end{equation*}
Write
\begin{equation*}
\det E_m^*=(n-3)\det\begin{bmatrix}M_r & 
B_1 & \beta_{m-1}\\ B_1^T & A_1 & \alpha_{m-1}\\ \beta_{m-1}^T & 
\alpha_{m-1}^T & n\end{bmatrix}+\det\tilde E_m^*
\end{equation*}
and
\begin{equation*}
\det E_m=(n-3)\det\begin{bmatrix}M_r & B_1 & \beta_m\\ B_1^T & A_1 & 
\alpha_m\\ \beta_m^T & \alpha_m^T & n\end{bmatrix}+\det\tilde E_m.
\end{equation*}
The first term on the right in $\det E_m^*$ is no larger than the first term
on the right in $\det E_m$, and we will see that the second term of $\det
E_m^*$ is strictly smaller than the second term of $\det E_m$.  By
subtracting row and column $m$ of $\tilde E_m$ from row and column $m-1$ of
$\tilde E_m$ and expanding by minors on column $m-1$ we find that
\begin{equation*}
\det\tilde E_m=(n-3)\det\begin{bmatrix}M_r & B_1 & \beta_m\\ B_1^T & 
A_1 & \alpha_m\\ \beta_m^T & \alpha_m^T & 3\end{bmatrix}.
\end{equation*}
On the other hand, using Corollary~\ref{coro:tildePosDef}, which implies
that $\tilde E_m^*$ is positive definite, and evaluating the determinant as
we did for $\det\tilde E_m$, we find that
\begin{equation*}
\det\tilde E_m^*<(n-3)\det\begin{bmatrix}M_r & B_1 & \beta_m\\ B_1^T & 
A_1 & \alpha_m\\ \beta_m^T & \alpha_m^T & 3\end{bmatrix}.
\end{equation*}
This follows because $\alpha_{m-1}-\alpha_m$ and $\beta_{m-1}-\beta_m$,
which together form the first $m-2$ elements of column $m-1$ in the
expansion by minors, are not both zero.
\end{proof}

\begin{corollary}
The matrix $A$ in Theorem~\ref{theo:duplCols} is a block matrix.
\end{corollary}
\begin{proof}
We have proved in Theorems~\ref{theo:minusOneAndThree}
and~\ref{theo:duplCols} that both of the conditions needed in
Lemma~\ref{lem:blockProperty} for $A$ to be a block matrix hold.
\end{proof}

Our conclusion is that, when $\det\tilde M_r>0$, the maximal determinant
completion $E_m^*$ of $M_r$ takes a form where $A$ is a block matrix with
some number $k$ of blocks whose sizes we will denote $b_1$, $b_2$, \ldots,
$b_k$, and where $B=\begin{bmatrix}B_1 & \ldots & B_k\end{bmatrix}$ with
each of the matrices $B_j$ a rank-1 matrix consisting of a column
$\beta_j^*$ repeated $b_j$ times.  This
establishes~(\ref{eqn:evensharperbound}).

\section*{Acknowledgements}
We gratefully acknowledge support 
from the Australian Research Council, 
the ARC Centre of Excellence for Mathematics and Statistics of
Complex Systems (MASCOS), and 
INRIA via the ANU-INRIA Associate Team ANC.


\begin{thebibliography}{99}

\bibitem{Bar33}
G. Barba,
Intorno al teorema di Hadamard sui determinanti a valore massimo
\emph{Giorn. Mat. Battaglini} \textbf{71} (1933), 70--86.

\bibitem{www37}
R. P. Brent, W. P. Orrick, J. H. Osborn and P. Zimmermann, 
Some D-optimal designs of orders 19 and 37,
\url{http://wwwmaths.anu.edu.au/~brent/maxdet/}

\bibitem{BrRy49}
R. H. Bruck and H. J. Ryser,
The nonexistence of certain finite projective planes
\emph{Canadian J. Math.} \textbf{1} (1949) 88--93. 

\bibitem{ChKoMo87}
T. Chadjipantelis, S. Kounias and C. Moyssiadis,
The maximum determinant of $21 \times 21$ $(+1,-1)$-matrices and
D-optimal designs,
\emph{J. Statist. Plann. Inference} \textbf{16}, 2 (1987), 167--178.

\bibitem{Charalambides05}
A. Charalambides, 
personal communication to W. Orrick, 2005.   

\bibitem{Cohn00}
J. H. E. Cohn,
Almost D-optimal designs,
\emph{Utilitas Math.} \textbf{57} (2000), 121--128.

\bibitem{ConSlo98}
N. J. A. Sloane and J. H. Conway, \emph{Sphere Packings, Lattices and Groups, volume 290 of Grundlehren der Mathematischen Wissenschaften}, Springer, New York--Berlin--Heidelberg, 3rd edition, 1998. 

\bibitem{Craigen90}
R. Craigen,
The range of the determinant function on the set of $n \times n$
$(0,1)$-matrices, 
\emph{J. Combin. Math. Combin. Comput.} \textbf{8} (1990), 161--171.

\bibitem{Ehlich64a}
H. Ehlich,	
Determinantenabsch\"atzungen f\"ur bin\"are {M}atrizen,
\emph{Math. Z.} 	
\textbf{83} (1964), 123--132.

\bibitem{Ehlich64b}
H. Ehlich,	
Determinantenabsch\"atzungen f\"ur bin\"are Matrizen 
mit $n \equiv 3 \bmod 4$,
\emph{Math. Z.} 	
\textbf{84} (1964), 438--447.

\bibitem{GeSe79}   
A. V. Geramita and J. Seberry,
\emph{Orthogonal Designs: Quadratic Forms and Hadamard Matrices},
Marcel Dekker, New York, 1979.

\bibitem{Hadamard93}
J. Hadamard,
R\'esolution d'une question relative aux d\'eterminants,
\emph{Bull. des Sci. Math.} \textbf{17} (1893), 240--246.

\bibitem{Horadam07}
K. J. Horadam,	
\emph{Hadamard Matrices and their Applications},
Princeton University Press, 2007.

\bibitem{McKay79}
B. D. McKay,
Hadamard equivalence via graph isomorphism,
\emph{Discrete Mathematics} \textbf{27} (1979), 213--214.

\bibitem{nauty}
B. D. McKay,  
\emph{nauty}, \url{http://cs.anu.edu.au/~bdm/nauty/}.

\bibitem{Metropolis69}
N. Metropolis, Spectra of determinant values in $(0, 1)$ matrices. 
In A. O. L. Atkin and B. J. Birch, editors, 
\emph{Computers in Number Theory: Proceedings of
the Science Research Atlas Symposium No.~2 held at Oxford, 18--23
August, 1969}, Academic Press, London, 1971, 271--276.

\bibitem{MoKo82}
C. Moyssiadis and S. Kounias, 
The exact D-optimal first order saturated design with $17$ observations,
\emph{J. Statist. Plann. Inference} \textbf{7} (1982), 13--27.

\bibitem{Orrick05} 
W. P. Orrick, 
The maximal {$\{-1,1\}$}-determinant of order $15$,
\emph{Metrika} \textbf{62} (2005), 195--219.

\bibitem{Orrick08a} 
W. P. Orrick, 
Switching operations for {H}adamard matrices,
\emph{SIAM J.\ Discrete Math.} \textbf{22} (2008), 31--50.

\bibitem{Orrick-spectrum-talk}
W. P. Orrick,
Range and distribution of determinants of binary matrices,
invited talk
presented at the \emph{Maximal Determinant Workshop} held at
the Australian National University, 17 May 2010.
Available from 
\url{http://mypage.iu.edu/~worrick/talks.html}. 

\bibitem{OrSo07}
W. P. Orrick and B. Solomon,
Large-determinant sign matrices of order $4k+1$, 
\emph{Discrete Math.} \textbf{307} (2007), 226--236.

\bibitem{maxdet}
W. P. Orrick and B. Solomon,
\emph{The Hadamard maximal determinant problem},
\url{http://www.indiana.edu/~maxdet/}.

\bibitem{spectrum}
W. P. Orrick and B. Solomon,
\emph{Spectrum of the determinant function},
\url{http://www.indiana.edu/~maxdet/spectrum.html}.

\bibitem{OS-19-R3}
W. P. Orrick and B. Solomon, 
\emph{Conjectured $19 \times 19\ \{-1, +1\}$ matrices of maximal determinant},
15 May 2003. 
\url{http://www.indiana.edu/~maxdet/d19.html}.

\bibitem{Osborn02}
J. H. Osborn,		
\emph{The Hadamard Maximal Determinant Problem},
{Honours thesis}, University of Melbourne, 2002.
Available from
\url{http://wwwmaths.anu.edu.au/~osborn/publications/pubsall.html}.

\bibitem{Smith88}
W. D. Smith,	 	
\emph{Studies in Computational Geometry Motivated by Mesh Generation},
{PhD dissertation}, Princeton University, 1988.

\bibitem{Tamura06}
H. Tamura,    
D-optimal designs and group divisible designs, \emph{J. Combin. Des.} \textbf{14} (2006), 451--462.


\bibitem{Zivkovic06}
M. \v{Z}ivkovi\'{c}, 	
Classification of small $(0,1)$ matrices, 
\emph{Linear Algebra Appl.} \textbf{414} (2006), 1, 310--346. 

\end{thebibliography}
\end{document}